\documentclass[11pt,a4paper]{amsart}

\newtheorem{theorem}{Theorem}[section]
\newtheorem{lemma}[theorem]{Lemma}

\theoremstyle{definition}
\newtheorem{definition}[theorem]{Definition}

\newtheorem{proposition}[theorem]{Proposition}
\theoremstyle{remark}
\newtheorem{remark}[theorem]{Remark}

\numberwithin{equation}{section}

\newcommand{\G}{{\cal G}}
\newcommand{\C}{{\mathbb C}}

\renewcommand{\a}{\ensuremath{\alpha}}
\renewcommand{\l}{\ensuremath{\lambda}}

\newcommand{\g}{\ensuremath{\mathfrak{G}}}

\renewcommand{\b}{\ensuremath{\beta}}

\newcommand{\Z}{{\mathbb Z}}
\newcommand{\N}{{\mathbb N}}
\newcommand{\R}{{\mathbb R}}
\renewcommand{\aa}{\ensuremath{\mathfrak{A}}\xspace}

\newcommand{\V}{\ensuremath{\mathcal{V}}\xspace}

\def\mh{\mathfrak{h}}

\def\sl{\mathfrak{sl}}

\def\l{\lambda}

\def\End{\text{End}}

\def\aa{\mathtt{a}}
\def\d{\mathtt{d}}
\def\g{\mathtt{g}}
\def\V{\mathcal{V}}
\def\G{\mathcal{G}}

\def\A{\mathcal{A}}
\def\B{\mathcal{B}}

\begin{document}

\title{Generalized Polynomial modules over the Virasoro algebra}


%    author one information
\author{Genqiang Liu}
\address{School of Mathematics and Statistics, Henan University, Kaifeng 475004, China. }
\curraddr{}
\email{liugenqiang@amss.ac.cn}
\thanks{}

%    author two information
\author{Yueqiang Zhao}
\address{School of Mathematics and Statistics, Henan University, Kaifeng 475004, China.}
\curraddr{}
\email{yueqiangzhao@163.com}
\thanks{}

%    \subjclass is required.

\date{}

\dedicatory{}

%    "Communicated by" -- provide editor's name; required.
\commby{Kailash C. Misra}

%    Abstract is required.
\begin{abstract} Let $\mathcal{B}_r$ be the $(r+1)$-dimensional quotient Lie algebra of
 the positive part of the Virasoro algebra $\V$. Irreducible $\B_r$-modules were used to construct irreducible Whittaker modules in \cite{MZ2}
 and  irreducible weight modules with
 infinite dimensional weight spaces over $\V$ in \cite{LLZ}.
 In the present paper, we construct  non-weight Virasoro modules $F(M, \Omega(\l,\beta))$ from
 irreducible $\B_r$-modules $M$ and $(\A,\V)$-modules $\Omega(\l,\beta)$.
We give necessary and sufficient conditions for the  Virasoro module $F(M, \Omega(\l,\beta))$ to
be irreducible. Using the weighting functor introduced by J. Nilsson, we also give
the isomorphism criterion for two $F(M, \Omega(\l,\beta))$.
\end{abstract}

\maketitle

\vskip 10pt

\section{Introduction}
The Virasoro algebra $\V$ is the algebra named after the physicist Miguel Angel Virasoro.
 It is a complex Lie algebra, given as the universal central extension  
of the derivation Lie algebra of the Laurent polynomial algebra  $\A=\C[x^{\pm 1}]$, see \cite{KR} and references therein.
More precisely,  $\V$  is the complex Lie algebra spanned by $\{\mathtt{c}, \d_i:i\in\mathbb{Z}\}$
with the following relations:
\begin{displaymath}
[\d_i,\d_j]=(j-i) \d_{i+j}+\delta_{i,-j}\frac{i^3-i}{12}\mathtt{c};\quad
[\d_i,\mathtt{c}]=0,  \,\forall\, i,j\in\mathbb{Z} .
\end{displaymath}

The Virasoro algebra $\V$ is  an
important algebra, which is widely used in conformal field theory and string theory.
The representation theory on the Virasoro algebra has  attracted a
lot of attentions from mathematicians and physicists, see \cite{CDK,FMS,GO,KR,K}. 
The classical part of the weight representation
theory of $\V$ is well developed, we can refer to the  recent monograph
\cite{IK} for a detailed survey. In \cite{Mt}, Mathieu showed that
 any simple Harish-Chandra $\V$ is a highest/lowest weight module  or a
 intermediate series module.
In \cite{MZ1}, it is further  shown that any
simple weight $\V$-module with a nonzero finite dimensional weight
space is a Harish-Chandra module. 

The study of  simple
  $\V$-modules which are not Harish-Chandra modules is also very popular. One of the earliest
weight modules with infinite dimensional weight spaces
 is the tensor products of some
highest weight module and some intermediate series module, see 
\cite{Zh}. The irreducibility of these tensor products
was  solved completely by Chen-Guo-Zhao, see \cite{CGZ}. In 2001, Conley and
Martin \cite{CM} defined another class of such modules. The irreducibility of these modules was determined in
\cite{LZ}.  In \cite{LLZ}, a large class of irreducible Virasoro module with
 infinite dimensional weight spaces
 were constructed from $\B_r$-modules, where $\B_r$ (see section 2) is an $(r+1)$-dimensional quotient Lie algebra
of  the positive part of the Virasoro algebra.

Non-weight
 Virasoro modules were also  studied by many researchers. These modules include Whittaker modules, see \cite{OW,LGZ,LZ2,FJK,Ya,MW,OW2, MZ2},
$\C[\d_0]$-free modules, see
\cite{N1,TZ1,TZ2}, highest-weight-like modules, see \cite{GLZ}, irreducible modules from Weyl modules, see \cite{LZ2}.
Whittaker modules and  more general modules were
described in a uniform way in \cite{MZ2}, using irreducible $\B_r$-modules.
In \cite{TZ2}, it was showed that if $M$ is a $\V$-module which is a free $\C[\d_0]$-module  of rank $1$,
 then $M\cong \Omega(\lambda,\beta)$ for some $\beta\in \C$ and $ \lambda\in \C\setminus\{0\}$. One can refer to Example \ref{E2} for the definition
 of $\Omega(\lambda,\beta)$.

It is easy to see that an irreducible $\V$-module is not a weight module if and only if it is $\C[\d_0]$-free.  So
 it is valuable to construct non-weight $\V$-modules from $\Omega(\lambda,\beta)$. This is the  motivation of the  present paper.

Now we briefly
describe the  main content of the present paper. In Sect.2, we first recall the definition of $(\A,\V)$-modules  from the reference \cite{D}.
The module $\Omega(\lambda,\beta)$ is a typical example of  an $(\A,\V)$-module. For any irreducible module $M$
over $\B_r$ and   an
$(\A,\V)$-module $W$,   we define a
Virasoro module structure on the vector space $F(M,
W)=M\otimes W$, see (\ref{def3}) and (\ref{GM}). These modules are generalization of Virasoro modules $\mathcal{N}(M, \a)$ defined in \cite{LLZ}.
 In Sect.3, we prove that the Virasoro module
$F(M,\Omega(\lambda,\beta))$ is reducible if and only if $M \cong M_\b$,
 a $1$-dimensional $\B_r$-module determined by the parameter $\b\in\C$, see Theorem \ref{mt}. Thus we
obtain a huge class of irreducible non-weight Virasoro modules.
 In \cite{N2}, Nilsson introduced a weighting functor which maps non-weight modules to weight modules.
 It is written in \cite{N2} that the idea of this functor is due to O. Mathieu. Using the weighting functor, we also
determine  necessary and sufficient conditions for two such
irreducible Virasoro modules $F(M,\Omega(\lambda,\beta))$ to be isomorphic, see Theorem \ref{IT}.

\section{Constructing new  Virasoro modules }

We denote by $\mathbb{Z}$, $\mathbb{Z}_+$, $\N$, $\R$ and
$\mathbb{C}$ the sets of  all integers, nonnegative integers,
positive integers, real numbers and complex numbers, respectively.

Firstly we recall the definition of an $(\A,\V)$-module which is a very nature notion,
 see Sect.3 in \cite{D}.

\begin{definition} An  $(\A,\V)$-module $M$ is  a  module both
for the Lie algebra $\mathcal{V}$  and the commutative associative algebra $\A$ with  compatible actions:
$$mx^{n+m}v=\d_n x^m v-x^m \d_nv,\,\  \mathtt{c}v=0,$$
where $m,n\in \Z, v\in M$.
\end{definition}

\

The following two examples are two interesting classes  of $(\A,\V)$-modules.

\

 \noindent{\bf{Example 1.}}{\label{E1} For each $\a,\b\in \C$, there is a natural $(\A,\V)$-module structure on $\A$ as follows:
$$\d_m x^n=(n+\a+\b m) x^{n+m}, \ x^m x^n=x^{n+m},\  \mathbf{c}x^n=0.$$ We denote this module by
$A(\a, \b)$ which is called the intermediate
series module. It is well known that, as a $\V$-module  $A(\a, \b)$ is reducible if and only if $\a\in\Z$, and $\b=0$ or $1$.
Mathieu showed that any irreducible uniformly bounded weight module over $\V$ is isomorphic to some
irreducible sub-quotient of $A(\a, \b)$, see \cite{Mt}.}

\

 \noindent{\bf{Example 2.}}\label{E2}  For $\lambda\in \C\setminus\{0\},\beta\in \C,$ denote by $\Omega(\lambda,\beta)=\C[t]$
the polynomial associative algebra over $\C$ in indeterminate $t$.
In \cite{LZ2}, a class of $(\A,\V)$-modules is defined  on
$\Omega(\lambda,\beta)$ by
 $$\mathtt{c}f(t)=0, \d_mf(t)=\lambda^m(t-\b m)f(t-m),$$
$$x^mf(t)=\l^m f(t-m),$$
 for all $m\in \Z, f(t)\in \C[t]$. From \cite{LZ2} we know that
$\Omega(\lambda,\beta)$ is irreducible when it was restricted to $\V$ if and only if $\beta\ne 0$. When $\b=0$, one can check that
$\C[t]t$ is a proper Virasoro submodule of $\Omega(\l, 0)$.
It was showed that if $M$ is a $\V$-module which is a free $\C[\d_0]$-module  of rank $1$,
 then $M\cong \Omega(\lambda,\beta)$ for some $\beta\in \C$ and $ \lambda\in \C\setminus\{0\}$, see \cite{TZ2}.

\

For an $(\A,\V)$-module $M$, we define the following operator on $M$
$$\g(m)=x^{-m}\d_m,\, m\in \Z.$$ We can check that
\begin{equation}\label{gr}\mathtt{g}(m)\mathtt{g}(k)v-\mathtt{g}(k)\mathtt{g}(m)v= -k\mathtt{g}(k)v+m\mathtt{g}(m)v+(k-m)\mathtt{g}(m+k)v,
\end{equation}
\begin{equation}\label{gr2}\g(m)x^nv-x^n\g(m)v=nx^nv,
\end{equation}
 for any $m, k,n\in \Z$, $v\in M$.

\

By abuse of language, we denote $\G$ by  the Lie algebra with the basis $\{\g(m):m\in\mathbb{Z}\}$
and the Lie bracket defined as follows:
\begin{equation}\label{LG}
[\mathtt{g}(m),\mathtt{g}(k)]=-k\mathtt{g}(k)+m\mathtt{g}(m)+(k-m)\mathtt{g}(m+k),\, \forall\, m,k\in \Z.
\end{equation}

Let us define the notion of $(\A,\G)$-modules which appears  naturally in the above construction.
\begin{definition} An  $(\A,\G)$-module $M$ is  a  module both
for the Lie algebra $\G$  and the commutative associative algebra $\A$ with  compatible actions:
$$\g(m)x^nv-x^n\g(m)v=nx^nv,$$
where $m,n\in \Z, v\in M$.
\end{definition}

From (\ref{gr}) and (\ref{gr2}), any $(\A,\V)$-module can be viewed as an $(\A,\G)$-module.
Conversely, an $(\A,\G)$-module $M$ can be  also viewed an $(\A,\V)$-module via
$$\d_mv=x^m\g(m)v,\, \mathtt{c}v=0,\, \forall\, v\in M.$$

\

Let $M$ be a $\G$-module  and $W$ be an $(\A,\V)$-module. Since $W$ is also a
$\G$-module, considering the tensor product $M\otimes W$ of $\G$-modules $M$ and $W$,
 there is a natural $(\A,\G)$-module structure on
$M\otimes W$ as follows
\begin{equation}\label{def1} \g(m) (v\otimes w)=v\otimes(\g(m)w)+(\g(m)v) \otimes w,\end{equation}
\begin{equation}\label{def2}x^m (v\otimes w)= v\otimes (x^mw),\end{equation}
where $m\in\Z, v\in M, w\in W$.

\

From $\d_m=x^m\g(m)$, we know that the action of $\V$ on $M\otimes W$ is
\begin{equation}\label{def3} \d_m (v\otimes w)=v\otimes(d_mw)+(\g(m)v) \otimes x^mw,\end{equation}
\begin{equation}\label{def4}\mathtt{c}(v\otimes w)=0,\end{equation}
where $m\in\Z, v\in M, w\in W$.

\

Consequently, the formulas  (\ref{def2}),(\ref{def3}) and (\ref{def4}) define an  $(\A,\V)$-module structure on $M\otimes W$.
We denote this $(\A,\V)$-module by $F(M, W)$.

\

Denote by $\V_+$ the Lie subalgebra of $\V$
spanned by all $\d_i$ with $i\geq 0$. For $r\in\mathbb{Z}_+$, denote
by $\V^{(r)}_+$ the Lie subalgebra of $\V$
generated by all $\d_i$, $i>r$, and by $ \B_r$ the quotient
algebra $\V_+/\V^{(r)}_+$. By $\bar \d_i$ we
denote the image of $\d_i$ in $ \B_r$. Note that  $ \B_r$
is a solvable Lie algebra of dimension $r+1$.

Let $M$ be a $\G$-module. Motivated by the notion of polynomial modules in \cite{BZ}, we suppose that the action of
 $\g(m)$ on $M$ is defined by
 $$\g(m)=\sum_{i=0}^r\frac{m^{i+1} \aa_i }{(i+1)!},$$ where $\aa_i\in \End_\C(M)$ for $i:0\leq i\leq r$.
From the Lie bracket (\ref{LG}) of $\G$, we can check that
$$[\aa_i,\aa_j]=(j-i)\aa_{i+j},\, \forall\, i,j:0\leq i, j\leq r,$$
where $\aa_{i+j}=0$, when $i+j> r$.
Thus $M$  can be viewed as a  module over the Lie algebra $\B_r$ via
$$\bar \d_i v=\aa_i v,\ \forall\, v\in M.$$

Conversely, for a module $M$ over $\B_r$, we define
the action of $\G$ on $M$ by
 \begin{equation}\label{GM}\mathtt{g}(m)v=\sum_{i=0}^r\frac{m^{i+1}\bar \d_i v}{(i+1)!},\, \forall\, v\in M.\end{equation}
\

\begin{lemma} Let $M$ be a module over $\B_r$. Then $M$ becomes a $\G$-module under the action (\ref{GM}).
\end{lemma}
\begin{proof} We can compute that
$$\aligned &\ \mathtt{g}(m)\mathtt{g}(k)v-\mathtt{g}(k)\mathtt{g}(m)v -m\mathtt{g}(m)v+k\mathtt{g}(k)v\\
=&\ \Big(\sum_{i=0}^r\frac{m^{i+1}}{(i+1)!}\bar \d_i\Big)
\Big(\sum_{j=0}^r\frac{k^{j+1}}{(j+1)!}\bar \d_j\Big)v\\ &
-\Big(\sum_{i=0}^r\frac{k^{j+1}}{(j+1)!}\bar \d_j\Big)
\Big(\sum_{i=0}^r\frac{m^{i+1}}{(i+1)!}\bar
\d_i\Big)v\\
&+\Big(k\sum_{i=0}^r\frac{k^{i+1}}{(i+1)!}
-m\sum_{i=0}^r\frac{m^{i+1}}{(i+1)!}\Big)\bar \d_iv\\
=&\ \sum_{i,j=0}^r\frac{m^{i+1}k^{j+1}}{(i+1)!(j+1)!}(j-i)\bar
\d_{i+j}v+\Big(k\sum_{i=0}^r\frac{k^{i+1}}{(i+1)!}\\ &
-m\sum_{i=0}^r\frac{m^{i+1}}{(i+1)!}\Big)\bar \d_iv\\
=&\ \left(k\sum_{i,j=0}^r\frac{m^{i+1}k^j}{(i+1)!j!}\bar \d_{i+j}
-m\sum_{i,j=0}^r\frac{m^ik^{j+1}}{i!(j+1)!}\bar
\d_{i+j}\right)v\\
&\ +\Big(k\sum_{i=0}^r\frac{k^{i+1}}{(i+1)!}
-m\sum_{i=0}^r\frac{m^{i+1}}{(i+1)!}\Big)\bar \d_iv\\
=&\ k\sum_{i=0}^r\sum_{j=0}^i\frac{m^{i+1-j}k^j}{(i+1-j)!j!}\bar
\d_iv-m\sum_{i=0}^r\sum_{j=0}^i\frac{k^{i+1-j}m^j}{(i+1-j)!j!}\bar
\d_{i}v\\
&\ +\Big(k\sum_{i=0}^r\frac{k^{i+1}}{(i+1)!}
-m\sum_{i=0}^r\frac{m^{i+1}}{(i+1)!}\Big)\bar \d_iv\endaligned$$
 $$\aligned
=&\ k\sum_{i=0}^r\sum_{j=0}^{i+1}\frac{m^{i+1-j}k^j}{(i+1-j)!j!}\bar
\d_{i}
-m\sum_{i=0}^r\sum_{j=0}^{i+1}\frac{k^{i+1-j}m^j}{(i+1-j)!j!}\bar
\d_{i}v\\ =&\ (k-m)\sum_{i=0}^r\frac{(m+k)^{i+1}}{(i+1)!}\bar \d_{i}v=(k-m)\mathtt{g}(m+k)v.\endaligned
$$
 The lemma is proved.
\end{proof}

For a module $M$ over $\B_r$,  we can see that the Virasoro
 module ${F}(M, W)$ is a weight module if and only if $W$ is a weight Virasoro module.
 From (\ref{def3}) and (\ref{GM}), we can see that if $W=\oplus_{\l\in\C}W_\lambda$ is a weight module over $\V$,
 then ${F}(M, W)=\oplus_{\l\in \C}{N}_{\l}$ is a weight Virasoro module
with  weight spaces ${N}_{\l}=M\otimes W_\l$ where
$$ {N}_{\l}=\{v\in {F}(M,
W)\,\,|\,\,d_0v=\l v\}.$$

\begin{proposition}\label{ff}
If $M_1,M_2$ are modules over $\B_r$, then as $(\A, \V)$-modules,
$$F(M_1,F(M_2,W))\cong F(M_1\otimes M_2,W).$$
\end{proposition}

\begin{proof}  For any $v_1\in M_1, v_2\in M_2, w\in W, m,n\in \Z$, we have that
$$\aligned  & \d_m\Big(v_1\otimes(v_2\otimes w)\Big)\\
=&\ v_1\otimes \d_m(v_2\otimes w)+(\mathtt{g}(m)v_1)\otimes(v_2\otimes x^m w)\\
=&\ v_1\otimes (v_2\otimes \d_mw+\mathtt{g}(m)v_2\otimes x^mw)+(\mathtt{g}(m)v_1)\otimes(v_2\otimes x^m w)\\
=&\ (v_1\otimes v_2)\otimes \d_mw+\Big(\mathtt{g}(m)(v_1\otimes v_2)\Big)\otimes x^mw\\
=&\ \d_m\Big((v_1\otimes v_2)\otimes w\Big),
\endaligned$$
and
$$x^m\Big(v_1\otimes(v_2\otimes w)\Big)=\ v_1\otimes  v_2\otimes x^mw=x^m\Big((v_1\otimes v_2)\otimes w\Big).$$
The proposition is proved.
\end{proof}

\begin{remark} From Proposition \ref{ff}, it is very hard to determine the irreducibility of $F(M,W)$ for
any irreducible $(\A,\V)$-module $W$, since it is a difficult problem to discuss the irreducibility of the tensor
product of two $\B_r$-modules. Moreover, up to now there is no explicit classification of all irreducible
$(\A,\V)$-modules. However, it is still  valuable to research  the structure of $F(M,W)$ for some
interesting $(\A,\V)$-modules.
\end{remark}

Let $M$ be a module over $\B_r$. From Example 1, (\ref{def3}) and (\ref{GM}), the action of $\V$ on ${F}(M,A(\a,\beta))$ is defined by
\begin{equation}\label{AA} \d_m (v\otimes x^n)=\Big((n+\a+\b m)v+\sum_{i=0}^r\frac{m^{i+1}\bar \d_i v}{(i+1)!}\Big)\otimes x^{n+m},\end{equation}
where $m,n\in \Z, v\in M$. Clearly  ${F}(M,A(\a,\beta))$ is a weight module over $\V$,
which is isomorphic to a module $\mathcal{N}(M, \a)$ defined in \cite{LLZ} where $M$ is modified by the action of $\d_0$.
 Thus our Virasoro modules ${F}(M, W)$ are generalization of the modules $\mathcal{N}(M, \a)$ define in \cite{LLZ}.

If $M$ is a finite dimensional irreducible $\B_r$-module, then by Lie's Theorem,
$M=\C v$, $ \bar \d_i M=0$ for any $i\in\N$ and $\bar \d_0 v= \gamma v$ for some $\gamma\in\C$.
We denote this $\B_r$-module by $M_\gamma$. Clearly
$${F}(M_\gamma,A(\a,\beta))\cong A(\a,\beta+\gamma).$$

The following Theorem was given  by Liu-Lu-Zhao, see Theorem 4 and Theorem 5 in \cite{LLZ}.
\begin{theorem}\label{LLZ}The following statements hold.
\itemize\item[(a).] Let $M$ be an irreducible $\B_r$-module.
Then the Virasoro module $F(M, A(\a,\beta))$ is reducible if and only if $M\cong M_\gamma$ and $\a\in\Z, \b+\gamma=0$ or $1$ .
\item[(b).] Let  $M, M'$ be infinite dimensional irreducible module over $\B_r$, and $\a, \a'$, $\b,\b'\in \C$.
Then $F(M,A(\a,\beta))\cong F(M',A(\a',\beta'))$ if and only if $M\cong M'$, $\a-\a'\in \Z$ and
$\b=\b'$.
\end{theorem}

\section{Irreducibility  and the isomorphism classes of $F(M,\Omega(\lambda,\beta))$ }

In this section, we will determine the irreducibility and the isomorphism
classes of the Virasoro modules $F(M,\Omega(\lambda,\beta))$ for any irreducible
$\B_r$-module $M$.

\subsection{Irreducibility of $F(M,\Omega(\lambda,\beta))$}

From (\ref{def3}) and (\ref{GM}), the $\V$-module structure on $F(M,\Omega(\lambda,\beta))$ is given by
 \begin{equation}\label{OA}\d_m (v\otimes f(t))=v\otimes \l^m(t-m\b)f(t-m)+\sum_{i=0}^r \frac{m^{i+1}\bar \d_iv}{(i+1)!} \otimes \l^mf(t-m).\end{equation}

 \begin{lemma}\label{trivial}Let $M$ be an irreducible module over $\B_r$. Then
 either $\bar \d_rM=0$ or the action of $\bar \d_r$ on $M$ is
bijective.\end{lemma}
\begin{proof}It is straightforward to check that $\bar \d_rM$ and
${\rm ann}_M(\bar \d_r)=\{v\in M|\bar \d_r v=0\}$ are submodules of
$V$. Then the lemma follows from the simplicity of $M$.
\end{proof}

For any $n\in\Z_+, m\in \Z$, let $h_m^0=1$ and $$h_m^n= \prod_{j=m+1}^{m+n}(t-j),\forall\ n>0.$$

It is clear that  $\{h_m^n\mid n\in \Z_+\}$ forms a
 basis of $\Omega(\lambda,\b)$ for any $m\in \Z$. Moreover
 \begin{equation}\label{HD}h^n_{m}-h^n_{m+1}=nh^{n-1}_{m+1}.\end{equation}
From the definition of $\Omega(\lambda,\beta)$, we have that
\begin{equation}\label{H}\d_mh^n_k=\l^m(t-m\beta)h^n_{k+m},\ \ x^mh^n_k=\l^mh^n_{k+m},\end{equation}
for any $m,k\in\Z, n\in\Z_+$.

\begin{theorem}\label{mt} Let  $M$ be an irreducible module over
$\B_r$, $\l\in \C\setminus \{0\}$.
 Then the Virasoro module
$F(M,\Omega(\lambda,\beta))$ is reducible if and only if $M\cong M_\b$.
\end{theorem}

\begin{proof}
If $M$ is finite dimensional, then $M\cong M_\gamma$ for some $\gamma\in \C$. In this case, we have that
$$\d_m(v\otimes f(t))=v\otimes (t-m(\b-\gamma))f(t-m),$$ for any $f(t)\in\Omega(\lambda,\beta)$. Hence
$$F(M_\gamma,\Omega(\lambda,\beta))\cong \Omega(\lambda,\beta-\gamma).$$
Thus $F(M_\gamma,\Omega(\lambda,\beta))$
is reducible if and only if $\b-\gamma = 0$, i.e., $\b=\gamma$.

\

Next we assume that $M$ is infinite dimensional. By Lemma \ref{trivial}, there exists an
 $r_1\in \N$ with $r_1\leq r$ such that the action of $\bar \d_{r_1}$ on $M$ is
injective and $\bar \d_{i}M=0$ for $i>r_1$. We may simply assume that $r_1=r$.

 Let $N$ be a nonzero submodule of $F(M,\Omega(\lambda,\beta))$.
 Since $$F(M,\Omega(\lambda,\beta))\cong F(M,\Omega(1,\beta)),$$ see Theorem
 \ref{IT} below,
we  may further assume that $\l=1$.

\

\noindent{\bf Claim 1.} If $u=\sum_{n=0}^lv_n\otimes h^n_0$ is a nonzero element
in $N$ with $v_l\neq 0$, then $\sum_{n=0}^l (\bar \d_{r}^2 v_n)\otimes h^n_m\in N$ for any $m\in \Z$ .

By (\ref{OA}) and (\ref{H}),
$$ \d_m (v_n\otimes h^n_0)=v_n\otimes(t-m\b)h^n_m+ (\mathtt{g}(m)v_n) \otimes h^n_m.$$
%where $g(m)=\sum_{i=0}^r \frac{m^{i+1}\bar d_i}{(i+1)!}$.
Consequently
$$\aligned   &\d_k\d_{m-k}(v_n\otimes h^n_0)\\
=&\ \d_k\left( v_n\otimes (t-(m-k)\b)h^n_{m-k}+\mathtt{g}(m-k)v_n\otimes h^n_{m-k} \right)\\
=&\ v_n\otimes (t-k\b)(t-(m-k)\b-k)h^n_{m}+\mathtt{g}(m-k)v_n\otimes(t-k\b) h^n_{m} \\
&\ + \mathtt{g}(k)v_n\otimes (t-(m-k)\b-k)h^n_{m}+\mathtt{g}(k)\mathtt{g}(m-k)v_n\otimes h^n_{m}
\endaligned$$

Since $k$ is an arbitrary integer, considering the coefficient of $k^{2r+2}$ in $\d_k\d_{m-k}u$,
 we obtain that $\sum_{n=0}^l (\bar \d_{r}^2 v_n)\otimes h^n_m\in N$ for $m\in \Z$. Then Claim 1 follows.

 \

\noindent{\bf Claim 2.} If $u=\sum_{n=0}^lv_n\otimes h^n_0$ is a nonzero element
in $N$, then $(\bar \d_{r}^2 v_l)\otimes 1\in N$. Consequently
 $(\bar \d_{r}^2 v_l)\otimes \Omega(\l,\b)\subset N$.

\

Note that $h^k_m\in \C[t]$ can be expressed as a polynomial in $m$ with degree $k$,
and the highest term is a constant (i.e., a degree $0$ polynomial in $t$).
From Claim 1, $$u(m):=\sum_{n=0}^l (\bar \d_{r}^2 v_n)\otimes h^n_m\in N,$$ for any $m\in \Z$.
Considering the coefficient of $m^l$ in $u(m)$, we know that $(\bar \d_{r}^2 v_l)\otimes 1\in N$.
 From $\d_0((\bar \d_{r}^2 v_l)\otimes 1)=(\bar \d_{r}^2 v_l)\otimes t$,
we have that  $(\bar \d_{r}^2 v_l)\otimes \Omega(\l,\b)\subset N$. Then Claim 2 follows.

\

Since the action of $\bar \d_r$ on $M$ is
injective, $(\bar \d_{r}^2 v_l)\neq 0$.
From Claim 2, there exists nonzero element $v\in M$ such that $v\otimes \Omega(\l,\b)\subset N$.

\

\noindent{\bf Claim 3.} If $v\otimes \Omega(\l)\subset N$,
 then $(\bar\d_i v)\otimes \Omega(\l)\subset N$ for any $i\in \Z_+$.

From $\d_m(v\otimes h^n_{k-m})=v\otimes (t-m\b)h^n_k+\Big(\sum_{i=0}^r \frac{m^{i+1}\bar \d_i v}{(i+1)!}\Big)\otimes h^n_k$,
we see that $\Big(\sum_{i=0}^r \frac{m^{i+1}\bar \d_i v}{(i+1)!}\Big)\otimes h^n_k\in N$, for any $k,m\in \Z, n\in\Z_+$.
Hence $(\bar \d_i v)\otimes h^n_k\in N$ for any $i=0,1,\dots,r$.

\

Since $M$ is an irreducible $\B_r$-module, from Claim 3, we obtain that $M\otimes \Omega(\l)= N$. Therefore
$F(M,\Omega(\lambda,\beta))$ is irreducible when $M$ is infinite dimensional.
\end{proof}

\subsection{Isomorphism of $F(V,\Omega(\lambda,\beta))$}

We will first recall the weighting functor
introduced in \cite{N2}.

For $a\in\C$,  let $I_a$ be the maximal ideal of  $\C[\d_0]$ generated by $\d_0-a$.
For a $\V$-module $M$ and  $n\in\Z$, let
$$M_{n}:= M/I_{n}M,\ \ \  \mathfrak{W}(M):=\oplus_{n\in\Z}( M_{n}\otimes x^n).$$
By Proposition 8 in \cite{N2}, we have the following construction.
\begin{proposition} The vector space $\mathfrak{W}(M)$  becomes a weight $\V$-module   under the following action:
\begin{equation}\label{3.3}\d_m\cdot((v+I_{n}M)\otimes x^n):= (\d_mv+I_{n+m}M)\otimes x^{n+m}.\end{equation}
\end{proposition}

We first establish the following useful lemma.

\begin{lemma} We have $\mathfrak{W}(\Omega(\l,\b))\cong A(0,1-\b)$.
\end{lemma}

\begin{proof} It is easy to see that $\dim(\Omega(\l,\b)/I_n(\Omega(\l,\b)))=1$ for any $n\in \Z$. Let
$v_n=1+I_n(\Omega(\l,\b))\in \Omega(\l,\b)/I_n(\Omega(\l,\b))$. We see that
$$\aligned \d_mv_n=&\lambda^m(t-m\b)+I_{m+n}(\Omega(\l,\b))\\
=&\lambda^m(m+n-m\b)+I_{m+n}(\Omega(\l,\b))\\
=&\lambda^m(n+m(1-\b))v_{m+n}.
\endaligned$$
Set $w_n=\l^nv_n$. Then $\d_m w_n=(n+m(1-\b))w_{m+n}$.
Thus the lemma follows.
\end{proof}

We know   that $\Omega(\l,1)$ is irreducible,
 and $A(0,0)$ is reducible as $\V$-modules.
  Thus the weighting functor $\mathfrak{W}$ does not map irreducible modules to irreducible modules.

\begin{proposition}\label{pp} As Virasoro modules, we have the isomorphism $$\mathfrak{W}(F(M, \Omega(\l,\b))\cong F(M,A(0,1-\b)).$$
\end{proposition}
\begin{proof} Note that $F(M, \Omega(\l,\b))=M\otimes \Omega(\l,\b)$. For any $n\in\Z$, using (\ref{OA}), we have
$$I_n(F(M, \Omega(\l,\b)))=I_n(M\otimes \Omega(\l,\b))=M\otimes I_n(\Omega(\l,\b)).$$
We can easily deduce  that $$\mathfrak{W}(F(M, \Omega(\l,\b))\cong F(M,A(0,1-\b)).$$
\end{proof}

Combining Proposition \ref{pp} with Theorem \ref{LLZ}, we obtain the following isomorphism criterion.

\begin{theorem}\label{IT} Let $M, M'$ be two infinite dimensional  irreducible $\B_r$-modules, $\l,\l'\in \C\setminus\{0\},
\b,\b'\in\C$.
 Then   $F(M, \Omega(\l,\b)\cong F(M', \Omega(\l',\b')$ if and only if $M\cong M'$,
 $\b=\b'$.
\end{theorem}

\begin{remark} From Theorem \ref{mt}, we can construct irreducible non-weight Virasoro modules from irreducible
$\B_r$-modules. All irreducible modules over $\B_1$ were
classified in \cite{Bl}, while all irreducible modules over $\B_2$ were
classified in \cite{MZ2}. The classification of irreducible modules over $\B_r$ remains open, for any  $r>2$.
\end{remark}



\bibliographystyle{amsplain}

\end{document}